\documentclass[12pt]{article}

\usepackage{amsmath,amssymb,amsthm}
\usepackage{xcolor,fullpage}
\usepackage{tikz}
\usepackage{complexity}
\usepackage{enumerate}
\usepackage{cite,url,hyperref}

\usepackage[final]{pdfpages}
\usepackage{subfigure}
\usepackage{comment}
\usepackage{algorithm}
\usepackage[noend]{algpseudocode}

\newtheorem{question}{Question}
\newtheorem{theorem}{Theorem}[section]
\newtheorem{lemma}[theorem]{Lemma}

\newtheorem{claim}[theorem]{Claim}

\usepackage[normalem]{ulem}

\newcommand{\qedclaim}{\hfill $\diamond$ \medskip}
\newenvironment{proofclaim}{\noindent\emph{Proof of Claim.}}{\qedclaim}

\DeclareMathOperator{\mex}{mex}
\newcommand{\xor}{\oplus}

\usepackage{todonotes}

\usepackage{authblk}
\title{The Normal Play of the Domination Game}

\author[1]{J. M. Brito}
\author[2]{T. Marcilon}
\author[3]{N. Martins}
\author[1]{R. Sampaio}
\affil[1]{Dept Computa\c c\~ao, Universidade Federal do Cear\'a, Fortaleza, Brazil}
\affil[2]{CCT, Universidade Federal do Cariri, Juazeiro do Norte, Brazil}
\affil[3]{Univ. Integ. Intern. Lusofonia Afro-Brasileira, Redenção, Brazil\thanks{email: \texttt{jmbmaluno@alu.ufc.br}, 
\texttt{thiago.marcilon@ufca.edu.br}, 
\texttt{nicolasam@unilab.edu.br}, \texttt{rudini@dc.ufc.br}}}

\date{}

\begin{document}
\maketitle
\begin{abstract}
In 2010, Bre\v{s}ar, Klav\v{z}ar and Rall introduced the optimization variant of the graph domination game and the game domination number, which was proved PSPACE-hard by Brešar et al. in 2016. In 2024, Leo Versteegen obtained the celebrated proof of the Conjecture $\frac{3}{5}$ on this variant of the domination game, proposed by Kinnersley, West and Zamani in 2013.
In this paper, we investigate for the first time the normal play of the domination game, which we call \textsc{Normal Domination Game}, that is an impartial game where the last to play wins.
We first prove that this game is PSPACE-complete even in graphs with diameter two.
We also use the Sprague-Grundy theory to prove that Alice (the first player) wins in the path $P_n$ if and only if $n$ is not a multiple of $4$, and wins in the cycle $C_n$ if and only if $n=4k+3$ for some integer $k$.
Moreover, we obtain a polynomial time algorithm to decide the winner for any disjoint union of paths and cycles in the \textsc{Normal Domination Game} and its natural partizan variant.
Finally, we also prove that the \textsc{Misère Domination Game} (the last to play loses) is PSPACE-complete, as are the natural partizan variants of the normal game and the misère game.
\end{abstract}


\section{Introduction}
\label{sec:intro}

Two of the most studied structures in graph theory are the independent sets and the dominating sets.
Given a graph $G$ and a set $S$ of vertices, we say that $S$ is \emph{independent} if there is no edge between the vertices of $S$ and we say that $S$ is a \emph{dominating set} if every vertex of $G$ is in $S$ or has a neighbor in $S$.
The \emph{independence number} $\alpha(G)$ and the \emph{domination number} $\gamma(G)$ are classic graph parameters, among the most investigated, consisting of the sizes of a maximum independent set and a minimum dominating set of the graph, respectively. 
These parameters were exhaustively investigated in the literature and are known to be NP-hard even in restricted graph classes.
In this paper, we deal with the game variants of these problems. 

The game variant of the independent set problem was introduced in 1908 by Henry Dudeney \cite{dudeney-1917} as the \textsc{Kayles} game in which two players, Alice and Bob, alternately select vertices of a graph $G$ in such a way that the set of selected vertices is an independent set.
The \textsc{Kayles} game is in the \emph{normal play convention}, that is, the last to play wins.
In the \emph{misère play convention}, the last to play loses and the game is called \textsc{Misère Kayles}.
The \textsc{Kayles} game was solved for some graph classes for the first time in 1956 by Guy and Smith \cite{guy56} through the Sprague-Grundy theory.
In 1978, Schaefer \cite{schaefer-1978} proved that the \textsc{Kayles} game is PSPACE-hard.
In 2002, Bodlander and Kratsch \cite{bodlaender02} proved that the \textsc{Kayles} game is polynomial time solvable in several graph classes, such as cographs, cocomparability graphs and circular arc graphs.
In 2001, the optimization variant of the \textsc{Kayles} game was introduced \cite{slater01}, in which Alice wants to maximize the number of selected vertices and Bob wants to minimize this number. The game independence number $\alpha_g(G)$ is the maximum number of selected vertices in the game with best play for Alice and Bob. Curiously, this optimization variant of the \textsc{Kayles} game and the game independence number $\alpha_g(G)$ have been little studied in the literature.

The game variant of the dominating set problem was introduced much later, in 2010, by Bre\v{s}ar, Klav\v{z}ar and Rall \cite{bresar10}, called as the \textsc{Domination Game}. In this game, Alice and Bob alternately select \emph{playable} vertices of a graph $G$, where a vertex is \emph{playable} in the game if it dominates at least one vertex that was not dominated before. The game ends when the selected vertices form a dominating set.
To the best of our knowledge, this game was investigated only in the \emph{optimization variant}, in which Alice wants to minimize the number of selected vertices and Bob wants to maximize this number. The \emph{game domination number} $\gamma_g(G)$ is the minimum number of selected vertices in the game with best play for Alice and Bob.
This parameter was intensively investigated.
In 2016, Brešar et al. proved that $\gamma_g(G)$ is PSPACE-hard \cite{bresar16}.
In 2013, Kinnersley, West, and Zamani \cite{west13} proved the Theorem $2/3$ and proposed the Conjecture $3/5$ on the \textsc{Domination Game}. They state that $\gamma_g(G)$ is at most $\frac{2}{3}n$ and $\frac{3}{5}n$, respectively, if $G$ has no isolated vertices.
Eleven years later, in 2024, Versteegen \cite{versteegen24} obtained the celebrated proof of the Conjecture $3/5$, which is now called the Theorem $3/5$, in 33 pages.

\paragraph{Our Contributions.}
In this paper, we start the investigation of the normal play of the domination game (the last to play wins), which we call \textsc{Normal Domination Game}.
We first prove that the \textsc{Normal Domination Game} is PSPACE-complete even in graph with diameter two. Our reduction is from the \textsc{Avoid-POSDNF-2} game, proved PSPACE-hard by Schaefer \cite{schaefer-1978}.
We also use the Sprague-Grundy theory to prove that Alice (the first player) wins in the path $P_n$ if and only if $n\%4\ne 0$ ($n$ is not a multiple of $4$), and wins in the cycle $C_n$ if and only if $n\%4=3$, where $\%$ is the remainder operation (mod).
We also determine in polynomial time the winner for any disjoint union of paths and cycles in the \textsc{Normal Domination Game} and its natural partizan variant, using the Combinatorial Game Theory.
Finally, we also prove that the \textsc{Misère Domination Game} is PSPACE-complete, as are the natural partizan variants of the normal game and the misère game.
We conclude with some open questions.

\section{Normal Domination game is PSPACE-hard}\label{sec:pspace}

In this section, we prove the PSPACE-completeness of the \textsc{Normal Domination Game}.
We obtain a reduction from the \textsc{Avoid-POSDNF-2} game, which was proved PSPACE-hard by Schaefer \cite{schaefer-1978}.
In this game, it is given a logical formula $\varphi$ in Disjunctive Normal Form\footnote{Disjunction of conjunctions, such as $(x_1\land x_4)\lor(x_2\land x_4)\lor(x_3\land x_4)\lor(x_5\land x_6)$} (DNF) such that every conjunction has at most two variables and every literal is positive. All variables start false. Alice and Bob alternate turns assigning true to a variable valued false. The player that makes the formula true after their move loses. That is, the first player that assigns true to a variable that is in a conjunction alone or in a conjunction with another already chosen variable loses.

We first show some properties of the \textsc{Avoid-POSDNF-2} game.

\begin{lemma}\label{lem.posdnf}
\textsc{Avoid-POSDNF-2} game is PSPACE-complete even if we assume that every variable appears in at least one conjunction, no variable appears in all conjunctions and, for each variable, the number of conjunctions containing it is odd.
\end{lemma}

\begin{proof}
If a variable $x_i$ does not appear in a conjunction, add the conjunction $(x_i\land x_i')$, where $x_i'$ is a new variable, and then $x_i$ or $x_i'$ can be selected without interfering with the other conjunctions, but, if $x_i$ (resp. $x_i'$) is selected when $x_i'$ (resp. $x_i$) was already selected, the player loses.
If some variable appears in all conjunctions, add a new variable $x$ and a new conjunction containing only $x$. With this, no variable appears in all conjunctions and no player wants to select $x$, since otherwise they lose.
If some variable appears in an even number of conjunctions, we can make the following. Add a new variable $x'$ and, for each variable that appears in an even number of conjunctions, add a conjunction with this variable and $x'$. Moreover, add three (resp. two) conjunctions with $x'$ alone if $x'$ appears in an even (resp. odd) number of conjunctions.
With this, each variable appears in an odd number of conjunctions. Moreover, no player wants to select $x'$ (since otherwise they lose) and then the new conjunctions do not change the outcome of the game.
\end{proof}

\begin{theorem}\label{thm-pspace-normal}
The \textsc{Normal Domination Game} is \emph{PSPACE}-complete even in graphs with diameter two.
\end{theorem}

\begin{proof}
Clearly, the problem is in PSPACE.
Let $\varphi$ be an instance of the \textsc{Avoid-POSDNF-2} game with variables $x_1,\ldots,x_n$ and conjunctions $C_1,\ldots,C_m$ satisfying Lemma \ref{lem.posdnf}.
We may assume that $m\geq 3$, since otherwise the \textsc{Avoid-POSDNF-2} game is very easy.
Let $G$ be the instance of the \textsc{Normal Domination Game} constructed in the following manner:

\begin{itemize}
\item For each conjunction $C_j$, create a vertex to represent each of its variables and let $Q^j$ be the set of these vertices. Let $Q = Q^1 \cup Q^2 \cup \ldots \cup Q^m$;
\item Let $V_i$ be the set of every vertex representing the variable $x_i$ in any set $Q^j$, and let each set $V_i$ be a clique. Notice that each $V_i$ has odd size since each $x_i$ appears in an odd number of conjunctions of $\varphi$;
\item For $1\leq j \leq m$, create the vertices $c_j$ and $\overline{c_j}$, add an edge between $c_j$ and every vertex of $Q^j$, and add an edge between $\overline{c_j}$ and every vertex of $Q \setminus Q^j$;
\item Let $C = \{c_1,\overline{c_1},c_2,\overline{c_2},\ldots,c_m,\overline{c_m}\}$ be a clique and let each $Q^j$ be an independent set.
\end{itemize}

Figure \ref{fig:reductnormdomgame} shows the constructed graph $G$ of the reduction for the formula $\varphi = (x_1 \land x_4) \lor (x_2\land x_4) \lor (x_3\land x_4) \lor (x_5 \land x_6)$. This reduction can be carried out in linear time in the size of $\varphi$.
Notice that $G$ has diameter two, since $m\geq 3$.

\begin{figure}
\centering
\begin{tikzpicture}[font=\footnotesize]
\tikzstyle{vertex} = [draw,circle,minimum size=16pt,inner sep=1pt]
\tikzstyle{vertexclause} = [draw,circle,fill=black,minimum size=8pt,inner sep=0.1pt]
\tikzstyle{clause} = [draw,blue,rectangle,rounded corners=5pt]
\node[vertexclause,label={[label distance=0.1mm]270:$x_1$}] (v11) at (-5.25,0){};
\node[vertexclause,label={[label distance=0.1mm]270:$x_4$}] (v14) at (-5.25,-1.2){};
\node[vertexclause,label={[label distance=0.1mm]270:$x_2$}] (v22) at (-1,0){};
\node[vertexclause,label={[label distance=0.1mm]270:$x_4$}] (v24) at (-1,-1.2){};
\node[vertexclause,label={[label distance=0.1mm]270:$x_4$}] (v34) at (3.25,-1.2){};
\node[vertexclause,label={[label distance=0.1mm]270:$x_3$}] (v33) at (3.25,0){};
\node[vertexclause,label={[label distance=0.1mm]270:$x_5$}] (v45) at (7.5,0){};
\node[vertexclause,label={[label distance=0.1mm]270:$x_6$}] (v46) at (7.5,-1.2){};

\node[vertex] (c1)  at (-5.25,5){$c_1$};
\node[vertex] (oc1) at (-4.05,5){$\overline{c_1}$};
\node[vertex] (c2)  at (-1.00,5){$c_2$};
\node[vertex] (oc2) at ( 0.20,5){$\overline{c_2}$};
\node[vertex] (c3)  at ( 3.25,5){$c_3$};
\node[vertex] (oc3) at ( 4.45,5){$\overline{c_3}$};
\node[vertex] (c4)  at ( 7.50,5){$c_4$};
\node[vertex] (oc4) at ( 8.70,5){$\overline{c_4}$};

\node[clause,minimum width=15.5cm,minimum height = 1cm,label=180:\Large{$C$}] (C) at (1.7,5){};

\node[clause,minimum width=1.2cm,minimum height = 2.5cm,label=270:\Large{$Q^1$}] (Q1) at (-5.25,-.7){};

\node[clause,minimum width=1.2cm,minimum height = 2.5cm,label=270:\Large{$Q^2$}] (Q2) at (-1,-.7){};

\node[clause,minimum width=1.2cm,minimum height = 2.5cm,label=270:\Large{$Q^3$}] (Q3) at (3.25,-.7){};

\node[clause,minimum width=1.2cm,minimum height = 2.5cm,label=270:\Large{$Q^4$}] (Q4) at (7.5,-.7){};

\draw[-] (v14) to (v24) to (v34) to[bend right=10] (v14);
\draw[-] (c1)  to (Q1);
\draw[-] (oc1) to (Q2);
\draw[-] (oc1) to (Q3);
\draw[-] (oc1) to (Q4);
\draw[-] (c2)  to (Q2);
\draw[-] (oc2) to (Q1);
\draw[-] (oc2) to (Q3);
\draw[-] (oc2) to (Q4);
\draw[-] (c3)  to (Q3);
\draw[-] (oc3) to (Q1);
\draw[-] (oc3) to (Q2);
\draw[-] (oc3) to (Q4);
\draw[-] (c4)  to (Q4);
\draw[-] (oc4)  to (Q1);
\draw[-] (oc4)  to (Q2);
\draw[-] (oc4)  to (Q3);
\end{tikzpicture}
\caption{\label{fig:reductnormdomgame} Graph $G$ of the reduction for $\varphi = (x_1 \land x_4) \lor (x_2\land x_4) \lor (x_3 \land x_4) \lor (x_5 \land x_6)$. $C$ is a clique, but $Q^1, Q^2$, $Q^3$ and $Q^4$ are independent sets. An edge to a set $Q^j$ means an edge to every vertex in $Q^j$.}
\end{figure}

During the \textsc{Normal Domination Game} on the constructed graph $G$, we say that the game is \emph{open} if there is a non-dominated vertex in every set $Q^j$ ($1 \leq j \leq m$). Otherwise, we say that the game is \emph{closed}. We say that a \emph{move closes the game} if it was open before the move and closed after the move.

\begin{claim}
\label{claim:Qnotdom}
If a move closes the game, there is a non-dominated vertex in $Q$ after this move.
\end{claim}

\begin{proofclaim}
Consider a move that closes the game. Recall that, before this move, there was a non-dominated vertex in every $Q^j$.
If the move is in $C$, there is a non-dominated vertex in $Q$ after this move, since no vertex in $C$ is adjacent to a vertex in every $Q^j$.
If the move is in $Q$, there is also a non-dominated vertex in $Q$ after this move, since no variable is in every conjunction of $\varphi$.
\end{proofclaim}

\begin{claim}
\label{claim:cantplayinC}
If a player chooses a vertex of $C$ when the game is open, they lose.
\end{claim}

\begin{proofclaim}
Suppose that the game is open and the player selects the vertex $c_j$ (resp. $\overline{c_j}$). The player cannot win with this move, since no vertex of $C$ is adjacent to a vertex in every $Q^\ell$. So, the opponent wins by selecting $\overline{c_j}$ (resp. $c_j$).
\end{proofclaim}

\begin{claim}
\label{claim:wincond}
If a player closes the game, they lose.
\end{claim}

\begin{proofclaim}
Consider a move that closes the game.
By Claim \ref{claim:Qnotdom}, there is at least one non-dominated vertex in $Q$ after this move, and then the player does not win with this move. However, after that, there is some set $Q^j$ whose vertices are all dominated. So, in the next turn, the opponent wins by selecting $\overline{c_j}$.
\end{proofclaim}

\begin{claim}
\label{claim:skip}
If the game is open and the player selects a vertex in $V_i$, where $V_i$ had an even number of unchosen vertices before this move, then the opponent can play on $V_i$ keeping the game open and the number of unchosen vertices in $V_i$ even.
\end{claim}

\begin{proofclaim}
Since the number of vertices in $V_i$ is odd and the number of unchosen vertices in $V_i$ is even, there is an odd number of chosen vertices in $V_i$. This means that there is at least one vertex in $V_i$ already chosen at this point of the game, which implies that all vertices in $V_i$ are already dominated by the chosen vertices.

Let $P_1$ be the player that selects a vertex in $V_i$ in the next turn and $P_2$ be the opponent of $P_1$. Since the game is open before $P_1$'s move, then no set $Q^j$ of size 1 contains a vertex of $V_i$ and every set $Q^j$ of size 2 containing a vertex of $V_i$ has a non-dominated vertex.
Thus, after $P_1$'s move on $V_i$, the game remains open because it does not dominate any new vertex in $Q$. Therefore, $P_1$ cannot win with this move.

After $P_1$'s move, there is an odd number of unchosen vertices in $V_i$. Let $v\in V_i$ be one such vertex, which is in some $Q^j$, and let $w$ be the other vertex in $Q^j$. Since the game is open and $v$ is dominated, $w$ is not dominated at this point. No vertex of $C$ was chosen previously, since, otherwise, $P_1$ could not have made its move on $V_i$. Thus, $c_j$ is not dominated. Therefore, the opponent can play in $v$, since it dominates a new vertex ($c_j$). Also, the game is still open, since $v$ does not dominate any new vertex in $Q$. Finally, since before $P_1$'s turn, there was an even number of unchosen vertices in $V_i$ and both players selected vertices in $V_i$, then, after $P_2$'s turn, there is still an even number of unchosen vertices in $V_i$.
\end{proofclaim}

\medskip

Let $P_1$ be either Alice or Bob and let $P_2$ be the other player. Suppose that $P_1$ has a winning strategy in the \textsc{Avoid-POSDNF-2} game. We show a winning strategy for $P_1$ in the \textsc{Normal Domination Game}. The strategy consists in choosing the corresponding vertices in $Q$ until the game closes. While the game is still open, no player wants to move in $C$; otherwise, according to Claim \ref{claim:cantplayinC}, the player loses. If $P_2$ chooses a vertex in some $V_i$ such that there is already a vertex in $V_i$ previously chosen, then, by Claim \ref{claim:skip}, $P_1$ can choose another vertex in $V_i$ so that the game remains open and the correspondence to the winning strategy in the \textsc{Avoid-POSDNF-2} game is maintained. Since $P_1$ has a winning strategy in the \textsc{Avoid-POSDNF-2} game, $P_2$ is the player that closes the game. Therefore, by Claim \ref{claim:wincond}, $P_1$ wins the \textsc{Normal Domination Game} on $G$.
\end{proof}

\section{Misère Domination Game is PSPACE-hard}\label{sec:path-cycle}

In this section, we prove the PSPACE-completeness of the \textsc{Misère Domination Game}, the variant of the domination game in which the last to play loses.

\begin{theorem}\label{thm-pspace-misere}
The \textsc{Misère Domination Game} is \emph{PSPACE}-complete even in graphs with diameter four.
\end{theorem}

\begin{proof}
Clearly, the problem is in PSPACE.
We obtain a reduction from the normal variant.
Let $H$ be an instance of the \textsc{Normal Domination Game}. Let $G$ be the graph obtained from $H$ by adding 6 new vertices $x_1,x_2,x_3,y_1,y_2,y_3$ and the edges $x_1x_2$, $x_1x_3$, $y_1y_2$, $y_1y_3$, $x_1v$ and $y_1v$ for every vertex $v$ of $H$. See Figure \ref{fig-dom-misere}.

\begin{figure}[ht]\centering
\scalebox{1.0}{
\begin{tikzpicture}[scale=1]
\tikzstyle{vertex}=[draw,circle,fill=black!10,minimum size=15pt,inner sep=1pt]

\node[vertex] (x1) at (1.0,1.0) {$x_1$};
\node[vertex] (y1) at (5.0,1.0) {$y_1$};
\node[vertex] (x2) at (0.0,1.7) {$x_2$};
\node[vertex] (y2) at (6.0,1.7) {$y_2$};
\node[vertex] (x3) at (0.0,0.3) {$x_3$};
\node[vertex] (y3) at (6.0,0.3) {$y_3$};

\draw (3,1) ellipse (0.5cm and 1cm);
\node at (3,1) {\large $H$};
\node at (1.7,2.3) {\LARGE $G$};
\draw (x2)--(x1)--(x3);
\draw (y2)--(y1)--(y3);
\path[-]
(x1)edge(2.4,1)edge(2.5,0.2)edge(2.5,1.8)edge(2.4,1.4)edge(2.4,0.6);
\path[-]
(y1)edge(3.6,1)edge(3.5,0.2)edge(3.5,1.8)edge(3.6,1.4)edge(3.6,0.6);
\end{tikzpicture}}
\caption{\label{fig-dom-misere}Proof of the PSPACE-hardness of the \textsc{Misère Domination Game}. Graph $G$ of the reduction from the \textsc{Normal Domination Game} on $H$.}
\end{figure} 

Suppose that some player $X\in\{Alice, Bob\}$ has a winning strategy in the \textsc{Normal Domination Game} on $H$ and let $Y$ be the opponent. We present a winning strategy for $X$ in the \textsc{Misère Domination Game} on $G$. The player $X$ always prefer to play the winning strategy of the \textsc{Normal Domination Game} on the same vertices of $H$ (which are also vertices of $G$), unless $Y$ plays outside $H$. Since $X$ wins the normal variant in $H$, then $Y$ is the first to play outside $H$.
If $y_1$ is the first vertex selected outside $H$, then $X$ selects $x_2$ and the only playable vertices are $x_1$ and $x_3$; thus $Y$ plays last and loses.
If $y_2$ is the first vertex selected outside $H$, then $X$ selects $x_1$ and the only playable vertices are $y_1$ and $y_3$; thus $Y$ plays last and loses. Analogously if the first vertex selected outside $H$ is $y_3$, $x_1$, $x_2$ or $x_3$.
\end{proof}

\section{Impartial game in paths and cycles}\label{sec:path-cycle}

In this section, we use the Sprague-Grundy theory to solve the \textsc{Normal Domination Game} in paths and cycles.
In order to do this, we need to explain some technical details of the Sprague-Grundy theory, which is focused on the normal play of impartial games such as the \textsc{Normal Domination Game}.

\subsection{Nimbers and the Sprague-Grundy Theory}\label{sec:nimbers}

We say that a two-person combinatorial game is \emph{impartial} if the set of moves available from any given position (or configuration) is the same for both players.
It is easy to see that the \textsc{Normal Domination Game} is impartial.

\textsc{Nim} is one of the most important impartial games in which two players take turns removing objects from disjoint heaps until there are no objects remaining.
On each turn, a player chooses a non empty heap and must remove a positive number of objects from this heap.
An instance of \textsc{Nim} is given by a sequence $\Phi=(h_1,h_2,\ldots,h_k)$, where $k$ is the number of heaps and $h_i\geq 1$ is the $i$-th heap's size.
\textsc{Nim} plays a fundamental role in the Sprague-Grundy game theory and was mathematically solved \cite{nim1901} by Bouton in 1901 for any instance $\Phi$. In other words, the decision problem of deciding which player has a winning strategy is polynomial time solvable in both \textsc{Nim} variants.
This is related to \textbf{nim-sum}$(\Phi)=h_1\xor\ldots\xor h_k$, where $\xor$ is the bitwise xor operation.
Alice has a winning strategy in the normal play of \textsc{Nim} if and only if \textbf{nim-sum}$(\Phi)>0$. Moreover, Alice has a winning strategy in the misère play of \textsc{Nim} if and only if \textbf{nim-sum}$(\Phi)>0$ and there is at least one heap with more than one object, or \textbf{nim-sum}$(\Phi)=0$ and all non-empty heaps have exactly one object.
The winning strategies in both variants is to finish every move with a nim-sum of 0, except in the misère variant when all non-empty heaps have exactly one object.

The Sprague-Grundy theory, independently developed by R. P. Sprague \cite{sprague36} and P. M. Grundy \cite{grundy39} in the 1930s, states that it is possible to associate a number (called \emph{nimber}) to every finite position of an impartial game, associating it to a one-heap game of \textsc{Nim} with that size. As expected from this association with heaps, after one move on a position with nimber $h>0$ of an impartial game, it is possible to obtain a position with any nimber in $\{0,1,\ldots,h-1\}$. 
One of the key elements of the Sprague-Grundy theory is that the nimber of a position can be calculated by the value $\mex\{h_1,\ldots,h_k\}$, where $\{h_1,\ldots,h_k\}$ contains all nimbers of the positions obtained after one move and the \emph{minimum excludant} $\mex$ is the minimum non-negative integer not included in the set.
Moreover, a position obtained by the disjoint union of $k$ positions with nimbers $h_1,\ldots,h_k$ has nimber $h_1\xor\ldots\xor h_k$.
Note that the operator $\xor$ is commutative and associative, and that $h_i\xor h_j=0$ if and only if $h_i=h_j$.

\subsection{Paths}\label{sec:paths}

In this section, we prove the following theorem regarding the  \textsc{Normal Domination Game} for paths $P_n$, where $P_n$ is the graph with vertices $n$ vertices $v_1,\ldots,v_n$ and $n-1$ edges $v_1v_2, v_2v_3,\ldots,v_{n-1}v_n$.
Let $n\ \%\ k$ be the remainder of the division of $n$ by $k$, also denoted by $n\mod k$.

\begin{theorem}\label{teo-nimb}
Consider the \textsc{Normal Domination Game} and let $n\geq 0$.
The nimber of $P_1$, $P_2$ and $P_3$ are 1, 1 e 2, respectively. Moreover, if $n\geq 4$, then the nimber of $P_n$ is 0, 1, 1 or 3 depending whether $n\% 4$ is 0, 1, 2 or 3, respectively.
Consequently, Bob wins on the path $P_n$ if and only if $n$ is multiple of 4.
\end{theorem}

As an application of this theorem, we can conclude by the Sprague-Grundy theory that Bob wins the \textsc{Normal Domination Game} on the disjoint union of the paths $P_3$, $P_6$ and $P_7$, since the total nimber is the bitwise xor operation of the nimbers of these graphs, which is $2\xor 1\xor 3=0$.


In order to prove this theorem, we define $P_n'$ in the \textsc{Normal Domination Game} as the path $P_{n+2}$ in which the first vertex is the only vertex already played (that is, a player selected it in the game and then it belongs to the final dominating set). Analogously, we define $P_n''$ as the path $P_{n+4}$ in which the first and the last vertices are the only ones already played. Note that $P_n$, $P_n'$ and $P_n''$ have exactly $n$ vertices that must be dominated in the game.

We first prove the following lemma regarding the nimbers of $P_n'$ and $P_n''$.

\begin{lemma}\label{lemma1}
The nimbers of $P_n'$ and $P_n''$ are equal to $n\% 4$.
\end{lemma}

\begin{proof}
The proof is by induction on $n$.
We first show the four base cases of $P_n''$ with $n=1,2,3,4$.
Since $P_1''$ is a $P_5$ $v_1v_2v_3v_4v_5$ where $v_1$ and $v_5$ were already selected, the three possible moves ($v_2,v_3,v_4$) win immediately (nimber $0$), and then the nimber of $P_1''$ is 1.
Since $P_2''$ is a $P_6$ $v_1\ldots v_6$ where $v_1$ and $v_6$ were already selected, there are four possible moves ($v_2$ to $v_5$), which win immediately (nimber $0$ in $v_3$ or $v_4$) or obtains a $P_1''$ (nimber 1 in $v_2$ or $v_5$), and then the nimber of $P_2''$ is $\mex\{0,1\}=2$.
Since $P_3''$ is a $P_7$ $v_1\ldots v_7$ where $v_1$ and $v_7$ were already selected, there are five possible moves ($v_2$ to $v_6$), which win immediately (nimber $0$ in $v_4$) or obtains a $P_1''$ (nimber 1 in $v_3$ or $v_5$) or obtains a $P_2''$ (nimber 2 in $v_2$ or $v_6$), and then the nimber of $P_3''$ is $\mex\{0,1,2\}=3$.
Finally, since $P_4''$ is a $P_8$ $v_1\ldots v_8$ where $v_1$ and $v_8$ were already selected, there are six possible moves ($v_2$ to $v_7$), which obtains a $P_1''$ (nimber 1 in $v_4$ or $v_5$) or obtains a $P_2''$ (nimber 2 in $v_3$ or $v_6$) or obtains a $P_3''$ (nimber 3 in $v_2$ or $v_7$), and then the nimber of $P_4''$ is $\mex\{1,2,3\}=0$.

Now we show the four base cases of $P_n'$ with $n=1,2,3,4$.
Since $P_1'$ is a $P_3$ $v_1v_2v_3$ where $v_1$ was already selected in the game, the two possible moves ($v_2$ and $v_3$) wins immediately (nimber $0$), and then the nimber of $P_1'$ is 1.
Since $P_2'$ is a $P_4$ $v_1v_2v_3v_4$ where $v_1$ was already selected, there are three possible moves ($v_2,v_3,v_4$), which obtains a $P_1'$ (with nimber 1 in $v_2$) or wins immediately (nimber $0$ in $v_3$ or $v_4$), and then the nimber of $P_2'$ is $\mex\{0,1\}=2$.
Since $P_3'$ is a $P_5$ $v_1v_2v_3v_4v_5$ where $v_1$ was already selected, there are four possible moves ($v_2,v_3,v_4,v_5$), which obtains a $P_2'$ (with nimber 2 in $v_2$) or obtains a $P_1'$ (with nimber 1 in $v_3$) or wins immediately (nimber $0$ in $v_4$) or obtains a $P_1''$ (with nimber 1 in $v_5$), and then the nimber of $P_3'$ is $\mex\{2,1,0,1\}=3$.
Finally, since $P_4'$ is a $P_6$ $v_1v_2v_3v_4v_5v_6$ where $v_1$ was already selected, there are five possible moves ($v_2,v_3,v_4,v_5,v_6$), which obtains a $P_3'$ (with nimber 3 in $v_2$) or obtains a $P_2'$ (with nimber 2 in $v_3$) or obtains a $P_1'$ (with nimber 1 in $v_4$) or obtains a $P_1''$ (with nimber 1 in $v_5$) or obtains a $P_2''$ (with nimber 2 in $v_6$), and then the nimber of $P_4'$ is $\mex\{3,2,1,1,2\}=0$.

So, let $n\geq 5$ and suppose by induction that the lemma holds for every value less than $n$.
First consider the game on $P_n'$ with vertices $v_0,v_1,\ldots,v_{n+1}$, where $v_0$ is the only vertex already selected. Let us enumerate the possible moves of the next player.
If the move is on $v_1$, $v_2$ or $v_3$, it obtains $P_{n-1}'$, $P_{n-2}'$ and $P_{n-3}'$, respectively, with nimbers $(n-1)\% 4$,  $(n-2)\% 4$ and $(n-3)\% 4$ by induction.
If the move is on $v_{n+1}$ or $v_n$, it obtains $P_{n-2}''$ and $P_{n-3}''$, respectively, with nimbers $(n-2)\% 4$ and $(n-3)\% 4$.
So, suppose that the first move is on $v_k$ for $k=4,\ldots,n-1$.
This move obtains two graphs, $P''_{k-3}$ and $P'_{n-k}$, with total nimber $r(n,k)=(k-3)\% 4\ \xor (n-k)\% 4$.
From Claim \ref{claim1} below, $r(n,k)$ is always different from $n\% 4$. Thus, the nimber of $P'_n$ is
\[
\mex\Big\{\ (n-1)\% 4,\ (n-2)\% 4,\ (n-3)\% 4\ \Big\}\ =\ n\% 4.
\]

Now consider the game on $P_n''$ with vertices $v_0,v_1,\ldots,v_{n+3}$, where $v_0$ and $v_{n+3}$ are the only selected vertices. Let us enumerate the possible moves of the next player.
If the move is on $v_1$, $v_2$ or $v_3$, it obtains $P_{n-1}''$, $P_{n-2}''$ and $P_{n-3}''$, respectively, with nimbers $(n-1)\% 4$,  $(n-2)\% 4$ and  $(n-3)\% 4$.
The same if the move is on $v_{n+2}$, $v_{n+1}$ or $v_n$.
So suppose that the move is on $v_k$ for $k=4,\ldots,n-1$.
It obtains two graphs $P''_{k-3}$ and $P''_{n-k}$, with total nimber $r(n,k)=(k-3)\% 4\ \xor (n-k)\% 4$, defined above. From Claim \ref{claim1} below, $r(n,k)$ is always different from $n\% 4$. Thus, the nimber of $P_n''$ is
\[
\mex\Big\{\ (n-1)\% 4,\ (n-2)\% 4,\ (n-3)\% 4\ \Big\}\ =\ n\% 4.
\]
\end{proof}

\begin{claim}\label{claim1}
Let $n$ and $k$ be integers. Then $r(n,k)=(k-3)\% 4\ \xor (n-k)\% 4$ is always different from $n\% 4$, where $\xor$ is the bitwise xor operation. Moreover, $r(n,k)$ is equal to $(n-1)\% 4$ if and only if $n\% 4\in\{1,3\}$ and $k\% 4\in\{0,2\}$.
\end{claim}

\begin{proof}
Notice that $r(n,k)=(k+1)\% 4\ \xor (n-k)\% 4$.
Table \ref{table1} shows all 16 possible values of $r(n,k)$ depending on the remainders of $n$ and $k$ modulo 4.
Notice that $r(n,k)$ is always different from $n\%4$ and that $r(n,k)$ is equal to $(n-1)\% 4$ if and only if $n\% 4\in\{1,3\}$ and $k\% 4\in\{0,2\}$.

\begin{table}[h!]\centering
\begin{tabular}{ |c|c|c|c|c| }
\hline
         & $k\%4=0$    & $k\%4=1$    & $k\%4=2$    & $k\%4=3$\\
\hline
$n\%4=0$ & $1\xor 0=1$ & $2\xor 3=1$ & $3\xor 2=1$ & $0\xor 1=1$\\
\hline
$n\%4=1$ & $1\xor 1=0$ & $2\xor 0=2$ & $3\xor 3=0$ & $0\xor 2=2$\\
\hline
$n\%4=2$ & $1\xor 2=3$ & $2\xor 1=3$ & $3\xor 0=3$ & $0\xor 3=3$\\
\hline
$n\%4=3$ & $1\xor 3=2$ & $2\xor 2=0$ & $3\xor 1=2$ & $0\xor 0=0$\\
\hline
\end{tabular}
\caption{Values of $r(n,k)$. Note that $r(n,k)\ne n\% 4$.}\label{table1}
\end{table}
\end{proof}

Now we are able to determine the nimbers of the paths $P_n$, using Lemma \ref{lemma1} on $P_n'$ and $P_n''$.

\begin{proof}[Proof of Theorem \ref{teo-nimb}]
Consider the \textsc{Normal Domination Game} on the path $P_n$ with vertices $v_1,v_2,\ldots,v_n$. Let us enumerate all possible moves of the next player.
If the move is on $v_1$ or $v_2$, it obtains $P_{n-2}'$ and $P_{n-3}'$, respectively, with nimbers $(n-2)\% 4$ and $(n-3)\% 4$ from Lemma \ref{lemma1}.
The same if the move is on $v_n$ or $v_{n-1}$.
So, suppose that the move is on $v_{k-1}$ for $k=4,\ldots,n-1$. This move obtains two graphs $P'_{k-3}$ and $P'_{n-k}$, with total nimber $r(n,k)=(k-3)\% 4\ \xor (n-k)\% 4$ from Lemma \ref{lemma1}. From Claim \ref{claim1}, we have that $r(n,k)$ is always different from $n\% 4$. Moreover, also from Claim \ref{claim1}, $r(n,k)$ is equal to $(n-1)\% 4$ if and only if $n\% 4\in\{1,3\}$ and $k\% 4\in\{0,2\}$.

Then, by applying the minimum excludent $\mex$, we obtain the following.
If $n\% 4\in\{0,2\}$, the nimber is $\min\{n\% 4, (n-1)\% 4\}$.
If $n\% 4\in\{1,3\}$, the nimber is $n\% 4$.
Thus, the nimber of $P_n$ is equal to 0, 1, 1 or 3 depending whether $n\% 4$ is 0, 1, 2 or 3.

Therefore, from the Sprague-Grundy theory, the second player (Bob) wins the \textsc{Normal Domination Game} on $P_n$ if and only if $n\% 4=0$, that is, $n$ is a multiple of $4$. 
\end{proof}

\subsection{Cycles}\label{sec:cycles}

As a consequence of Lemma \ref{lemma1}, we also solve the game for cycles.

\begin{theorem}\label{teo-nimb2}
Consider the \textsc{Normal Domination Game} and let $n\geq 3$. The nimber of the cycle $C_n$ is equal to 1 if $n\% 4=3$, and is equal to 0, otherwise. Consequently, Alice wins the game on $C_n$ if and only if $n\% 4=3$ (that is, $n=4k+3$ for some integer $k$).
\end{theorem}

\begin{proof}
Consider the \textsc{Normal Domination Game} on the cycle $C_n$ with vertices $v_1,v_2,\ldots,v_n$. 
Without loss of generality, we may assume that the first player selected the vertex $v_1$.
This move obtains a position equivalent to a $P_{n-3}''$, whose nimber is $(n-3)\% 4$ from Lemma \ref{lemma1}.
Then, by applying the minimum excludent $\mex$, we have that the nimber of $C_n$ is $\mex\{(n-3)\% 4\}$, which is 1, if $n\% 4=3$, and is 0, otherwise.

Therefore, from the Sprague-Grundy theory, the first player (Alice) wins the \textsc{Normal Domination Game} on $C_n$ if and only if $n\% 4=3$, that is, $n=4k+3$ for some integer $k$.
\end{proof}

As an application of Theorem \ref{teo-nimb2}, we can conclude from the Sprague-Grundy theory that the second player (Bob) wins the \textsc{Normal Domination Game} on the disjoint union of $C_3$, $C_4$, $C_5$, $C_6$ and $C_7$, since the total nimber is the bitwise xor operation of the nimbers of these graphs, which is $1\xor 0\xor 0\xor 0\xor 1=0$.

\begin{figure}[ht!] \centering
\scalebox{0.8}{
\begin{tikzpicture}[scale=1]
\tikzstyle{V}=[draw,circle,fill=black!00,inner sep=3pt, minimum size=3pt]
\tikzstyle{D}=[draw,circle,fill=black!20,inner sep=3pt, minimum size=3pt]

\node[V] (a) at (4,0) {};
\node[V] (b) at (4,1) {}; \draw (a)--(b);
\node[V] (a) at (4,2) {}; \draw (a)--(b);
\node[V] (b) at (4,3) {}; \draw (a)--(b);
\node[D] (x) at (4,4) {}; \draw (x)--(b);
\node[V] (a) at (3.5,5) {}; \draw (x)--(a);
\node[V] (b) at (2.5,5.5) {}; \draw (a)--(b);
\node[V] (a) at (1.5,5.7) {}; \draw (a)--(b);
\node[V] (a) at (4.5,5) {}; \draw (x)--(a);
\node[V] (b) at (5.5,5.5) {}; \draw (a)--(b);
\node[V] (a) at (3,4.5) {}; \draw (x)--(a);
\node[V] (b) at (2,4.5) {}; \draw (a)--(b);
\node[V] (a) at (1,4) {}; \draw (a)--(b);
\node[V] (b) at (2,3.5) {}; \draw (a)--(b);
\node[V] (a) at (3,3.5) {}; \draw (x)--(a); \draw (a)--(b);

\node[V] (a) at (5,4.5) {}; \draw (x)--(a);
\node[V] (b) at (5.5,4.5) {}; \draw (a)--(b);
\node[V] (a) at (6.0,4.5) {}; \draw (a)--(b);
\node[V] (b) at (6.5,4.5) {}; \draw (a)--(b);
\node[V] (a) at (7,4) {}; \draw (a)--(b);
\node[V] (b) at (6.5,3.5) {}; \draw (a)--(b);
\node[V] (a) at (6.0,3.5) {}; \draw (a)--(b);
\node[V] (b) at (5.5,3.5) {}; \draw (a)--(b);
\node[V] (a) at (5,3.5) {}; \draw (a)--(b); \draw (x)--(a);
\end{tikzpicture}}
\caption{\label{fig:dom-nimber1} Alice wins the \textsc{Normal Domination Game} in the gray vertex.}
\end{figure}

Figura \ref{fig:dom-nimber1} shows another example, where Alice wins by playing on the gray vertex in her first move, since this returns a position equivalent to the disjoint union of $P_1'$, $P_2'$, $P_3'$, $P_3''$ and $P_7''$, with total nimber equal to $1\xor 2\xor 3\xor 3\xor 3=0$.


\section{Partizan variants are PSPACE-complete}

In this section, we investigate for the first time the partizan variants of the \textsc{Domination Game} in normal play and misère play, which we call \textsc{Normal Partizan Domination Game} and \textsc{Misère Partizan Domination Game}.
As done in the literature of combinatorial games, a partizan variant of an impartial game can be obtained by modifying the instances in such a way that the set of possible moves of Alice is disjoint from the set of possible moves of Bob.
See for example the classical book \emph{Winning Ways for your Mathematical Plays} from Berlekamp, Conway and Guy \cite{conway82} and the games Hackenbush (impartial) and Blue-Red Hackenbush (partizan).

In the partizan variants of the domination game, the instance is a graph $G$ such that the vertices are colored either $A$ or $B$, and Alice (resp. Bob) can only select vertices colored $A$ (resp. $B$). As in the impartial variants, the players alternately select playable vertices. In normal (resp. misère) play, the last to play wins (resp. loses).

In this section, we prove the computational complexity of the partizan variants.

\begin{theorem}
The \textsc{Normal} (resp. \textsc{Misère}) \textsc{Partizan Domination Game} is \emph{PSPACE}-complete even in graphs with diameter two (resp. four).
\end{theorem}

\begin{proof}
Consider an impartial domination game on normal play or misère play on a graph $H$ with at least one edge. Let $G$ be the graph obtained from $H$ by adding, for every vertex $v$ of $H$, a new vertex $v'$ and the edges $vv'$, $uv'$ and $u'v'$ for every neighbor $u$ of $v$ in $H$. Color the vertices of $G$ in the following way: for every vertex $v$ of $H$, $v$ is colored $A$ and $v'$ is colored $B$. Notice that, if $v$ (resp. $v'$) is selected during the domination game, then $v'$ (resp. $v$) is not playable anymore, since $v$ (resp. $v'$) dominates the same vertices that can be dominated by $v'$ (resp. $v$). This means that the normal and misère partizan domination games behave exactly like the impartial domination games. Therefore, since the diameter of $G$ is the same of $H$, then the \textsc{Normal Partizan Domination Game} is PSPACE-complete even in graphs with diameter two from Theorem \ref{thm-pspace-normal} and the \textsc{Misère Partizan Domination Game} is PSPACE-complete even in graphs with diameter four from Theorem \ref{thm-pspace-misere}.
\end{proof}

\section{Partizan game in paths and cycles}

Here, we use Conway's Combinatorial Game Theory to solve the \textsc{Normal Partizan Domination Game} in paths and cycles. In order to do this, we need to explain some technical details of Conways's theory, which is focused on the normal play
of partizan games.

\subsection{Combinatorial Game Theory and the Notation \texorpdfstring{$\{X|Y\}$}{X|Y}}

Here we give a very short introduction to the Combinatorial Game Theory in order to help us applying it later.
In this subsection, a \emph{game} mean a position of a partizan game.
Every game $J$ is denoted by $J=\{\mathcal{J}^A|\mathcal{J}^B\}$ where $\mathcal{J^A}$ and $\mathcal{J}^B$ are the \emph{options} of Alice and Bob, respectively, that is, $\mathcal{J}^A$ (resp. $\mathcal{J}^B$) is the set of games that can be obtained after a move of Alice (resp. Bob).
For simplification, it is common to write $J=\{J^{A_1},\ldots,J^{A_m}|J^{B_1},\ldots,J^{B_n}\}$, without brackets, when $\mathcal{J}^A=\{J^{A_1},\ldots,J^{A_m}\}$ and $\mathcal{J}^B=\{J^{B_1},\ldots,J^{B_n}\}$.

The games below are important and are denoted by $0$, $1$, $-1$, $*$, $\uparrow$ and $\downarrow$.
In the empty game $0=\{|\}$, Alice and Bob have no moves and then the first to play loses.
In the game $1=\{0|\}$, Bob has no move and any move of Alice wins, leading to a game 0.
In the game $-1=\{|0\}$, Alice has no move and any move of Bob wins, leading to a game 0.
In the game star $*=\{0|0\}$, the first to play wins, leading to a game 0.
Moreover, $\uparrow=\{0|*\}$ and $\downarrow=\{*|0\}$.
The games $n$ and $-n$ can also be defined for every integer $n\geq 2$. In the game $n=\{n-1|\}$, Alice has $n$ more moves than Bob. In the game $-n=\{|-(n-1)\}$, Bob has $n$ more moves than Alice.

The Combinatorial Game Theory is a very rich theory with many results and definitions, such as infinitesimals, canonical form of a game, the Simplification Theorem for Numbers, Temperature Theory for hot games and so on.
Here the main points are the following. A game is equal to $0$ if and only if the first player loses, independently if the first player is Alice or Bob.
Moreover, the sum $J_1+J_2$ of two games is the game resulting of the disjoint union of the games $J_1$ and $J_2$, that is, the player chooses either $J_1$ or $J_2$ and make a move there.
Finally, every game $J=J_1+J_2$ has a unique canonical form $\{X|Y\}$, which can be determined in polynomial time given all options $J_1^A$ and $J_2^A$ of Alice and all options $J_1^B$ and $J_2^B$ of Bob.

The negative $-J$ of a game $J$ is the game in which the roles of Alice and Bob are changed. Notice that $J-J=J+(-J)=0$, since, for every move of the first player in one game, the opponent has the same move in the other game, which is then the last to play, winning the game.

\subsection{Partizan game in paths}

Let us apply the Combinatorial Game Theory on the \textsc{Normal Partizan Domination Game} in paths.
We describe a path as its sequence of colors. For example, $AABA$ and $ABAA$ represents the same path $P_4$ $v_1v_2v_3v_4$ where $v_1$, $v_2$ and $v_4$ are colored A and $v_3$ is the only vertex colored B. Moreover, $BBAB=-AABA$, since the roles of Alice and Bob are changed. Notice that $AABA+BBAB=0$, that is, the first player loses in the disjoint union of two $P_4$'s with these colors.

As an example, let us determine the values of all paths with at most 4 vertices. 
For the instances of the path $P_1$, notice that $A=1$ and $B=-1$.
For $P_2$, $AA=1$, $BB=-1$ and $AB=BA=\{0|0\}=*$. 

For $P_3$, $AAA=2$ and, $BBB=-2$.
Note that $AAB=BAA=\{\dot{A}B,0|A\dot{A}\}=\{*,0|1\}=\frac{1}{2}$, where $\dot{A}B$ means that the dotted vertex does not need to be dominated, but can be selected to dominate other vertices.
Moreover, $ABB=BBA=-\frac{1}{2}$, $ABA=\{*|0\}=\downarrow$ and $BAA=\{0|*\}=\uparrow$.

Finally, for $P_4$, $AAAA=3$, $BBBB=-3$, $AAAB=BAAA=\{1|2\}=\frac{3}{2}$, $ABBB=BBBA=-\frac{3}{2}$, $AABA=ABAA=\left\{\frac{1}{2}|1\right\}=\frac{3}{4}$, $ABAA=BBAB=-\frac{3}{4}$ and $AABB=BBAA=ABAB=BABA=ABBA=BAAB=0$.

Let $v_1,\ldots,v_n$ be a path $P_n$ with vertices colored either $A$ or $B$ and let $1\leq i\leq j\leq n$.
Let $P_{i,j}$ represent the subpath from $v_i$ to $v_j$.
Let $X_{i,j}$ represent the subpath from $v_i$ to $v_j$ considering that $v_i$ does not need to be dominated.
Let $Y_{i,j}$ represent the subpath from $v_i$ to $v_j$ considering that $v_j$ does not need to be dominated.
Let $Z_{i,j}$ represent the subpath from $v_i$ to $v_j$ considering that $v_i$ and $v_j$ do not need to be dominated.

Our objective is to determine the value of $P_{1,n}$, the entire instance.

\begin{lemma}\label{lem-partizan}
Let $n\geq 1$ and let $G$ be a graph $P_n$ with vertices $v_1,\ldots,v_n$ colored either $A$ or $B$. Then
$P_{i,j}=0$ if $i>j$, $X_{i,j}=Y_{i,j}=0$ if $i\geq j$ and $Z_{i,j}=0$ if $i\geq j-1$. Otherwise, 
$P_{i,j}=\{\mathcal{P}^A_{i,j}|\mathcal{P}^B_{i,j}\}$, $X_{i,j}=\{\mathcal{X}^A_{i,j}|\mathcal{X}^B_{i,j}\}$, $Y_{i,j}=\{\mathcal{Y}^A_{i,j}|\mathcal{Y}^B_{i,j}\}$ and $Z_{i,j}=\{\mathcal{Z}^A_{i,j}|\mathcal{Z}^B_{i,j}\}$, where, for $C\in\{A,B\}$:
\[
\mathcal{P}^C_{i,j}=\{Y_{i,k-1}+X_{k+1,j}:\ i\leq k\leq j\ \mbox{and}\ v_k\ \mbox{is colored}\ C\},
\]
\[
\mathcal{X}^C_{i,j}=\{Z_{i,k-1}+X_{k+1,j}:\ i\leq k\leq j\ \mbox{and}\ v_k\ \mbox{is colored}\ C\},
\]
\[
\mathcal{Y}^C_{i,j}=\{Y_{i,k-1}+Z_{k+1,j}:\ i\leq k\leq j\ \mbox{and}\ v_k\ \mbox{is colored}\ C\},
\]
\[
\mathcal{Z}^C_{i,j}=\{Z_{i,k-1}+Z_{k+1,j}:\ i\leq k\leq j\ \mbox{and}\ v_k\ \mbox{is colored}\ C\}.
\]
\end{lemma}

\begin{proof}
If $i>j$, then $P_{i,j}=X_{i,j}=Y_{i,j}=Z_{i,j}=0$, since there is no vertex in these instances.
If $i=j$, then $X_{i,j}=Y_{i,j}=Z_{i,j}=0$, since there is only one vertex in these instances and it is already dominated.
If $i=j-1$, then $Z_{i,j}=0$, since there is only two vertices in this instance and both are already dominated.

First consider the instance $P_{i,j}$ with $i\leq j$ and let $v_k$ be a vertex colored A in $P_{i,j}$, that is, $i\leq k\leq j$. Then the move of Alice in $v_k$ divides the subpath in two parts. The first part is from $v_i$ to $v_{k-1}$, where $v_{k-1}$ is already dominated (by $v_k$), but can be used to dominated other vertices. The second part is from $v_{k+1}$ to $v_j$, where $v_{k+1}$ is already dominated as before. Thus, the first part is represented by the game $Y_{i,k-1}$ and the second part by the game $X_{k+1,j}$, and both are independent, that is, the resulting game is a disjoint union of these games, which is represented by the sum $Y_{i,k-1}+X_{k+1,j}$. Analogous for Bob's moves.

Similarly, in the case of $X_{i,j}$, the parts are represented by the games $Z_{i,k-1}$ and $X_{k+1,j}$, since $v_i$, $v_{k-1}$ and $v_{k+1}$ are the only vertices already dominated.
In the case of $Y_{i,j}$, the parts are represented by the games $Y_{i,k-1}$ and $Z_{k+1,j}$, since $v_{k-1}$, $v_{k+1}$ and $v_j$ are the only vertices already dominated.
Finally, in the case of $Z_{i,j}$, the parts are represented by the games $Z_{i,k-1}$ and $Z_{k+1,j}$, since $v_i$, $v_{k-1}$, $v_{k+1}$ and $v_j$ are the only vertices already dominated.
\end{proof}

\begin{theorem}
The \textsc{Normal Partizan Domination Game} is polynomial time solvable in disjoint unions of paths and cycles.
\end{theorem}

\begin{proof}
First let us solve for any path $P_n$. 
Algorithm \ref{algoritmo1} below is a dynamic programming algorithm that applies Lemma \ref{lem-partizan} in order to determine the value of $P_{1,n}$, which is the value of the entire instance. If $P_{1,n}>0$, Alice wins. If $P_{1,n}<0$, Bob wins. If $P_{1,n}=0$, the second player wins. If $P_{1,n}\ ||\ 0$, the first player wins.

The lines 24 to 27 of Algorithm \ref{algoritmo1} calls the polynomial algorithm to obtain the canonical form of a game, given its options.
Thus, the time of this algorithm is polynomial.

For the cycle $C_{n+1}$ with vertices $v_1,\ldots,v_{n+1}$ colored either $A$ or $B$, we can determine all options of Alice (resp. Bob) by selecting one vertex colored A (resp. B), renaming the vertices in such a way that the selected vertex is $v_{n+1}$ and then determining the value of $Z_{1,n}$ for the path $v_1,\ldots v_n$, that is, considering that $v_1$ and $v_n$ are already dominated. This can be done by the algorithm above.
After obtaining all options of Alice and Bob, we can determine the value of the instance on the cycle $C_{n+1}$.

\end{proof}

\begin{algorithm}
\caption{Partizan-Normal-Dom-Paths}\label{algoritmo1}
\begin{algorithmic}[1]
\For{$i\leftarrow 1,\ldots,n+1$}
  \State $P_{i,i-1}\leftarrow 0$;\hspace{2cm}$Z_{i,i-1}\leftarrow Z_{i,i}\leftarrow Z_{i,i+1}\leftarrow 0$;
  \State $X_{i,i-1}\leftarrow X_{i,i}\leftarrow 0$;\ \ \ \ \ $Y_{i,i-1}\leftarrow Y_{i,i}\leftarrow 0$;
  \If{($i\leq n$) \textbf{and} ($v_i$ is colored $A$)}
    \State\ \ $P_{i,i}\leftarrow +1$
  \Else\ $P_{i,i}\leftarrow -1$
  \EndIf
  \If{($i\leq n-1$) \textbf{and} ($v_i$ and $v_{i+1}$ are colored $A$)}
    \State $P_{i,i+1}\leftarrow X_{i,i+1}\leftarrow Y_{i,i+1}\leftarrow 1$; 
  \Else\If{($i\leq n-1$) \textbf{and} ($v_i$ and $v_{i+1}$ are colored $B$)}
    \State\ \ $P_{i,i+1}\leftarrow X_{i,i+1}\leftarrow Y_{i,i+1}\leftarrow -1$;
  \Else\ $P_{i,i+1}\leftarrow X_{i,i+1}\leftarrow Y_{i,i+1}\leftarrow\ \ *$;
  \EndIf
  \EndIf
\EndFor
\For{$i\leftarrow n-2,\ldots,1$ (dec)}
  \For{$j\leftarrow i+2,\ldots,n$}
    \State $\mathcal{P}_{i,j}^A\ \leftarrow\ \mathcal{P}_{i,j}^B\ \leftarrow\ \emptyset$;\ \ \ \ \ \ $\mathcal{Z}_{i,j}^A\ \leftarrow\ \mathcal{Z}_{i,j}^B\ \leftarrow\ \emptyset$
    \State$\mathcal{X}_{i,j}^A\ \leftarrow\ \mathcal{X}_{i,j}^B\ \leftarrow\ \emptyset$;\ \ \ \ \ \ $\mathcal{Y}_{i,j}^A\ \leftarrow\ \mathcal{Y}_{i,j}^B\ \leftarrow\ \emptyset$
    \For{$k\leftarrow i,\ldots,j$}
      \If{$v_k$ is colored $A$}
        \State $\mathcal{P}^A_{i,j}\ \leftarrow\  \mathcal{P}^A_{i,j}\ \cup\ \{Y_{i,k-1}+X_{k+1,j}\}$
        \State $\mathcal{X}^A_{i,j}\ \leftarrow\  \mathcal{X}^A_{i,j}\ \cup\ \{Z_{i,k-1}+X_{k+1,j}\}$
        \State $\mathcal{Y}^A_{i,j}\ \leftarrow\  \mathcal{Y}^A_{i,j}\ \cup\ \{Y_{i,k-1}+Z_{k+1,j}\}$
        \State $\mathcal{Z}^A_{i,j}\ \leftarrow\  \mathcal{Z}^A_{i,j}\ \cup\ \{Z_{i,k-1}+Z_{k+1,j}\}$
      \EndIf
      \If{$v_k$ is colored $B$}
        \State $\mathcal{P}^B_{i,j}\ \leftarrow\  \mathcal{P}^B_{i,j}\ \cup\ \{Y_{i,k-1}+X_{k+1,j}\}$
        \State $\mathcal{X}^B_{i,j}\ \leftarrow\  \mathcal{X}^B_{i,j}\ \cup\ \{Z_{i,k-1}+X_{k+1,j}\}$
        \State $\mathcal{Y}^B_{i,j}\ \leftarrow\  \mathcal{Y}^B_{i,j}\ \cup\ \{Y_{i,k-1}+Z_{k+1,j}\}$
        \State $\mathcal{Z}^B_{i,j}\ \leftarrow\  \mathcal{Z}^B_{i,j}\ \cup\ \{Z_{i,k-1}+Z_{k+1,j}\}$
      \EndIf
    \EndFor
    \State $P_{i,j}\ \leftarrow\  \{\mathcal{P}^A_{i,j}\ |\ \mathcal{P}^B_{i,j}\}$;\ \ \ \ \ \ $Z_{i,j}\ \leftarrow\  \{\mathcal{Z}^A_{i,j}\ |\ \mathcal{Z}^B_{i,j}\}$
    \State $X_{i,j}\ \leftarrow\  \{\mathcal{X}^A_{i,j}\ |\ \mathcal{X}^B_{i,j}\}$;\ \ \ \ \ \ $Y_{i,j}\ \leftarrow\  \{\mathcal{Y}^A_{i,j}\ |\ \mathcal{Y}^B_{i,j}\}$
  \EndFor
\EndFor
\State \textbf{return} $P_{1,n}$
\end{algorithmic}
\end{algorithm}

\vspace{0.5cm}

\section{Conclusions and Open Problems}

In this paper, we investigated the normal variant of the \textsc{Domination Game}, which we call \textsc{Normal Domination Game}. 
The optimization variant of this game and the game domination number $\gamma_g(G)$ have many results in the literature. However, this is the first paper investigating the normal variant.

We prove the PSPACE-hardness of the \textsc{Normal Domination Game} and solve this game for paths and cycles.
We also investigated the normal variant of natural partizan version of this game, which we call \textsc{Normal Partizan Domination Game}.
We also prove the following:

\begin{theorem}
Let $G$ be a graph whose connected components are paths or cycles.
Then the \textsc{Normal Domination Game} on $G$ can be solved in logarithmic time.
Moreover, the \textsc{Normal Partizan Domination Game} on $G$ can be solved in polynomial time.
\end{theorem}

\begin{proof}
In the \textsc{Normal Domination Game}, it is possible to determine from Theorems \ref{teo-nimb} and \ref{teo-nimb2} the nimber of any path $P_n$ and any cycle $C_n$, based only on the number $n$ (size of $\lceil\log_2 n\rceil$ bits). Therefore, it is possible to determine the nimber of $G$ from the Sprague-Grundy theory by applying the bitwise-xor operation on the nimbers of the components of $G$, in logarithmic time.
If the nimber is $0$, the second player (Bob) wins. Otherwise, Alice wins.

Similarly for the \textsc{Normal Partizan Domination Game}, using Lemma \ref{lem-partizan}, obtaining a polynomial time.
\end{proof}

It would be interesting to determine a large class of graphs in which the game can be solved in polynomial time.
We were unfortunately unable to fully characterize the class of trees. For instance, 

\begin{question}
The \textsc{Normal Domination Game} is polynomial time solvable in trees?
\end{question}

Moreover, the Sprague-Grundy theory is a bit more complicated for misère games.

\begin{question}
The \textsc{Misère Domination Game} is polynomial time solvable in paths and cycles? And in trees?
\end{question}

Moreover, the Sprague-Grundy theory is a bit more complicated for misère games.

\begin{question}
The \textsc{Misère Partizan Domination Game} is polynomial time solvable in paths and cycles?
\end{question}

\bibliographystyle{plain}

\end{document}